\documentclass[12pt]{article}
\usepackage{amsmath, amsthm, amssymb, amsfonts, mathrsfs}
\usepackage{graphicx}
\usepackage{caption}
\usepackage{subcaption}
\usepackage{hyperref}
\usepackage{tikz}
\usepackage{tkz-berge}
\newtheorem{theorem}{Theorem}[section]

\newtheorem{lem}[theorem]{Lemma}
\newtheorem{prop}[theorem]{Proposition}

\newtheorem{ex}{Example}[section]

\def \Zl {{\mathbb Z}}
\def \Nl {{\mathbb N}}
\def \Rl {{\mathbb R}}
\def \Ql {{\mathbb Q}}
\def \Cl {{\mathbb C}}

\def \vl {{\mathbf v}}

\title{Quantum State Transfer on a Class of Circulant Graphs}

\author{ Hiranmoy Pal\\
Department of Mathematics\\
Indian Institute of Information Technology Bhagalpur\\Bhagalpur, India - 813210\\
Email: hiranmoy@iiitbh.ac.in
}
\date{\today}

\begin{document}
\maketitle

\vspace{-0.3in}

\begin{center}{Abstract}\end{center}
We study the existence of quantum state transfer on non-integral circulant graphs. We find that continuous time quantum walks on quantum networks based on certain circulant graphs with $2^k$ $\left(k\in\mathbb{Z}\right)$ vertices exhibit pretty good state transfer when there is no external dynamic control over the system. We generalize few previously known results on pretty good state transfer on circulant graphs, and this way we re-discover all integral circulant graphs on $2^k$ vertices exhibiting perfect state transfer.
\newpage

\section{Introduction}
We consider state transfer with respect to the adjacency matrix of a graph (see \cite{bon}). All graphs presented are assumed to be finite, undirected and simple. Let $G$ be a graph on $n$ vertices with adjacency matrix $A$. The transition matrix associated to $G$ is defined by
\begin{eqnarray}\label{IE1}
H(t):=\exp{\left(-itA\right)},~\text{where}~t\in\Rl.
\end{eqnarray}
A quantum state in a quantum network based on $G$ is defined to be a unit vector in $\Cl^n$. The Equation \ref{IE1} governs the way a quantum state evolve over time. Let ${\bf e}_u$ denote the characteristic vector corresponding to the vertex $u$ of $G$. The graph $G$ is said to exhibit pretty good state transfer (PGST) between a pair of vertices $u$ and $v$ if there is a sequence $\left\lbrace t_k\right\rbrace$ of real numbers such that
\begin{eqnarray}\label{IE2}
\lim_{k\rightarrow\infty} H(t_k) {\bf e}_u=\gamma {\bf e}_v,~\text{where}~\gamma\in\Cl~\text{and}~|\gamma|=1.
\end{eqnarray}
If there is a sub-sequential limit $t_0$ of $\left\lbrace t_k\right\rbrace$ then Equation \ref{IE2} gives $H(t_0) {\bf e}_u=\gamma {\bf e}_v,$ in which case we say that there is perfect state transfer (PST) (see \cite{bose}) between $u$ and $v$ at time $t_0$. If $H(t_0) {\bf e}_u=\gamma {\bf e}_u$ for $t_0\neq 0$ then $G$ is said to be periodic at $u$ at time $t_0$. Moreover $G$ is said to be periodic if it is periodic at all vertices at the same time (in that case $H(t_0)=\gamma I$). Perfect state transfer is most desirable in this context, however, Godsil \cite{god2} has shown that there are very few graphs exhibiting PST. Consequently the study of PGST, introduced by Godsil in \cite{god1}, acquired considerable attention. In \cite{god4}, Godsil \emph{et al.} characterized all paths exhibiting PGST between the end vertices. Later, Coutinho \emph{et al.} \cite{cou3} further investigated and found a class paths having PGST between internal vertices. In \cite{pal4}, we found that a cycle on $n$ vertices exhibits PGST if and only if $n$ is a power of $2$. This was the first examples of non-integral circulant graphs admitting PGST. Apart from cycles, we have also found families of non-integral circulant graphs exhibiting PGST in \cite{pal0, pal4}. Several other ralavent results regarding state transfer can be found in \cite{mil4, fan, kirk, pal, pal1, pal2, pal3, saxe}. We extend some pre-existing results appearing in \cite{pal0, pal4} on circulant graphs admitting PGST and re-discover families of circulant graphs on $2^k$ vertices exhibiting PST.\par
We begin with a brief introduction to Kronecker approximation theorem on simultaneous approximation of numbers. We apply this result to establish some relations on the eigenvalues of a circulant graph in Lemma \ref{L3} and Theorem \ref{T4}.  
\begin{theorem}[Kronecker approximation theorem]\cite{apo}
If $\alpha_1,\ldots,\alpha_l$ are arbitrary real numbers and if $1,\theta_1,\ldots, \theta_l$ are real, algebraic numbers linearly independent over $\Ql$ then for $\epsilon>0$ there exist $q\in\Zl$ and $p_1,\ldots,p_l\in\Zl$ such that
\[\left|q\theta_j-p_j-\alpha_j\right|<\epsilon.\]
\end{theorem}
A Cayley graph over a finite abelian group $\left(\Gamma,+\right)$ with the connection set $S$, where $0\notin S\subseteq\Gamma$ and $\left\lbrace -s:s\in S\right\rbrace=S$, is denoted by $Cay\left(\Gamma,S\right)$. The elements of $\Gamma$ are the vertices of the graph, where two vertices $a,b\in\Gamma$ are adjacent if and only if $a-b\in S$. If $\Gamma=\Zl_n$ then the Cayley graph is called circulant graph. In particular, a Cayley graph over $\Zl_n$ with $S=\left\lbrace1,n-1\right\rbrace$ is called a cycle which is denoted by $C_n$. The eigenvalues of $C_n$ are given by
\begin{eqnarray}\label{IE3}
\lambda_l=2\cos{\left(\frac{2l\pi}{n}\right)},\;0\leq l\leq n-1,
\end{eqnarray}
and the associated eigenvectors are $\vl_l=\left[1,\omega_n^l,\ldots,\omega_n^{l(n-1)}\right]^T$, where $\omega_n=\exp{\left(\frac{2\pi i}{n}\right)}$ is the primitive $n$-th root of unity.\par
Let the graphs $G$ and $H$ be defined on a common vertex set $V$, and the respective edge sets $E(G)$ and $E(H)$. The edge union of $G$ and $H$, denoted $G\cup H$, is a graph with vertex set $V$ and edge set $E(G)\cup E(H)$. The following result enables us to realise the transition matrix of a Cayley graph as a product of transition matrices of its sub-graphs.
\begin{prop}\cite{pal1}\label{IP1}
Let $\Gamma$ be a finite abelian group and consider two disjoint and symmetric subsets $S,T\subset\Gamma$. Suppose the transition matrices of $Cay(\Gamma, S)$ and $ Cay(\Gamma, T)$ are $H_{S}(t)$ and $H_{T}(t)$, respectively. Then $Cay(\Gamma, S\cup T)$ has the transition matrix $H_{S}(t)H_{T}(t).$
\end{prop}
A graph is called integral if all its eigenvalues are integers. Let $n,d\in\Nl$ and $d$ be a proper divisor of $n$. We denote $S_n(d)=\left\lbrace x\in\Zl_n: gcd(x,n)=d\right\rbrace.$ If $D$ is a set containing proper divisors of $n$ then define $S_n(D)=\bigcup\limits_{d\in D} S_n(d).$ The set $S_n(D)$ is called a gcd-set of $\Zl_n$. The following theorem due to W. So characterizes all integral circulant graphs.
\begin{theorem}\cite{so}\label{so}
A circulant graph $Cay\left(\Zl_n,S\right)$ is integral if and only if $S$ is a gcd-set.
\end{theorem}
Hence, for a non-integral circulant graph $Cay\left(\Zl_n,S\right)$, there must exist at least one such $d$ so that $S\cap S_n(d)$ is a non-empty proper subset of $S_n(d)$.\par
Now we include the proof the following proposition which was mentioned in \cite{fan} without a proof. The result gives a relation between PST and PGST in periodic graphs.
\begin{prop}\label{IP2}
If a graph is periodic, then it admits perfect state transfer if and only if it admits pretty good state transfer.
\end{prop}
\begin{proof}
It is easy to observe that if a graph $G$ has PST then it must also have PGST. On the other hand, suppose that a graph is periodic at $\tau$ and therefore for a vertex $u$ of $G$ there exists $\gamma\in\Cl$ with $|\gamma|=1$ so that $H(\tau){\bf e}_u=\gamma{\bf e}_u$. For $u,v\in V(G)$, the $uv$-th entry of $H(t)$ is a continuous function of real numbers. Also observe that, for $t\in\Rl$,
\begin{eqnarray*}
\left|{\bf e}_v^T H(t+\tau){\bf e}_u\right| = \left|{\bf e}_v^T H(t)H(\tau){\bf e}_u\right|
&=& \left|{\bf e}_v^T H(t)\left(\gamma{\bf e}_u\right)\right|,\text{ as }H(\tau){\bf e}_u=\gamma{\bf e}_u\\
&=& \left|{\bf e}_v^T H(t){\bf e}_u\right|.
\end{eqnarray*}
Therefore, the modulus of the $uv$-th entry of $H(t)$ is a periodic continuous function and so its image is a compact set. Since $H(t)$ is unitary, the modulus of each entry of $H(t)$ is bounded by $1$ for all $t\in\Rl$. Finally, if it satisfies Equation \ref{IE2} then by extreme value theorem there exists $\tau_0\in\Rl$ so that $|{\bf e}_v^T H(\tau_0){\bf e}_u|=1,$ \emph{i.e,} the graph $G$ has PST between $u$ and $v$. Hence the result follows.
\end{proof}
In the following section we introduce some previously known results on PGST in circulant graphs and after that we include the main results with some examples for convenience.

\section{Pretty Good State Transfer on Circulant Graphs}
We begin the discussion with few relevant results on circulant graphs exhibiting PGST appearing in \cite{pal0,pal4}. It is well known that if a graph $G$ admits PGST between two vertices $u$ and $v$ then all automorphisms of $G$ that fix $u$ must also fix $v.$ This leads to the conclusion which we present as follows.
\begin{lem}\label{L1}\cite{pal4}
If pretty good state transfer occurs in a circulant graph $Cay\left(\Zl_n,S\right)$ then $n$ is even and it occurs only between the pair of vertices $u$ and $u+\frac{n}{2}$, where $u,u+\frac{n}{2}\in\Zl_n$.
\end{lem}
A circulant graph $G$ being vertex transitive, for any pair of vertices $u$, $v$ in $G$ there exists an automorphism mapping $u$ to $v$. Let $A$ be the adjacency matrix of $G$. Note that the transition matrix $H(t)$ of $G$ can be realised as a polynomial in $A$. If $P$ is the matrix of an automorphism of $G$ then $P$ commutes with $A$ as well as $H(t)$. If $G$ exhibits PGST between two vertices $u$ and $v$ then Equation \ref{IE2} infers that
\[\lim_{k\rightarrow\infty} H(t_k) \left(Pe_u\right)=\gamma \left(Pe_v\right).\]
by Lemma \ref{L1}, it is enough to find PGST in a circulant graph between the pair of vertices $0$ and $\frac{n}{2}$. Let the spectral decomposition of the adjacency matrix of $C_n$ be $A=\sum\limits_{l=0}^{n-1}\lambda_lE_l,$ where $E_l=\frac{1}{n}\vl_l\vl_l^*$ and $\lambda_l,~\vl_l$ is as mentioned in Equation \ref{IE3}. Therefore, from Equation \ref{IE1}, the transition matrix of $C_n$ is evaluated as
\[H(t)=\exp{\left(-itA\right)}=\sum\limits_{l=0}^{n-1}\exp{\left(-i\lambda_l t\right)E_l}.\]
Note that $\left(0,\frac{n}{2}\right)$-th entry of $E_l$ is $\frac{1}{n}\omega_n^{-\frac{nl}{2}}$. Hence $(0,\frac{n}{2})$-th entry of $H(t)$ is given by 
\begin{eqnarray}\label{E1}
H(t)_{0,\frac{n}{2}}=\frac{1}{n}\sum\limits_{l=0}^{n-1}\exp{\left(-i\lambda_l t\right)}\cdot\omega_{n}^{-\frac{nl}{2}}=\frac{1}{n}\sum\limits_{l=0}^{n-1}\exp{\left[-i\left(\lambda_l t+l\pi\right)\right]}.
\end{eqnarray}
It is shown in \cite{saxe} that if a cycle admits PST then it must be integral and this leads to the conclusion that only $C_4$ exhibits PST (see \cite{mil4}). We further investigated and in \cite{pal4}, we have found a characterization of all cycles admitting PGST. 
\begin{theorem}\label{T1}\cite{pal4}
A cycle $C_n$ admits pretty good state transfer if and only if $n=2^k$ for $k\geq 2.$
\end{theorem}
 In \cite{wal1}, it appears that the eigenvectors of a Cayley graph over an abelian group are independent of the connection set. Therefore the set of eigenvectors of two Cayley graphs defined over the same abelian group can be chosen to be equal. Hence we have the following result appearing in \cite{pal1}.
\begin{prop}\cite{pal1}\label{IP3}
If $S_1$ and $S_2$ are symmetric subsets of an abelian group $\Gamma$ then adjacency matrices of the Cayley garphs $Cay(\Gamma, S_1 )$ and $Cay(\Gamma, S_2 )$ commute.
\end{prop}
Let the eigenvalues of a circulant graph $Cay(\Zl_n, S)$ be $\theta_l,$ for $l=0,1,\ldots,n-1$. If $H_S(t)$ is the transition matrix of $Cay(\Zl_n, S)$ then, by Proposition \ref{IP3}, we obtain
\begin{eqnarray}
H_S(t)_{0,\frac{n}{2}} &=&\frac{1}{n}\sum\limits_{l=0}^{n-1}\exp{\left[-i\left(\theta_l t+l\pi\right)\right]},\label{E2}\\
H_S(t)_{0,0}&=&\frac{1}{n}\sum\limits_{l=0}^{n-1}\exp{\left(-i\theta_l t\right)}.\label{E2a}
\end{eqnarray}
A graph $G$ is said to be almost periodic if there is a sequence $\left\lbrace t_k\right\rbrace$ of non-zero real numbers and a complex number $\gamma$ of unit modulus such that $\lim\limits_{k\rightarrow\infty} H(t_k) =\gamma I,$ where $I$ is the identity matrix of appropriate order. Since the circulant graphs are vertex transitive, we observe that a circulant graph is almost periodic if and only if $\lim\limits_{k\rightarrow\infty} H(t_k) {\bf e}_0=\gamma {\bf e}_0.$ This implies that if a cycle admits PGST then it is necessarily almost periodic. Next we find a connection between the time sequences of two cycles one of which exhibits PGST and the other one is almost periodic.

\begin{lem}\label{L2}\cite{pal0}
Let $k,k'\in \Nl$ and $2\leq k'<k$. If the cycle $C_{2^k}$ admits pretty good state transfer with respect to a sequence $\left\lbrace t_m\right\rbrace$ then $C_{2^{k'}}$ is almost periodic with respect to a sub-sequence of $\left\lbrace t_m\right\rbrace$. 
\end{lem}

Apart from the cycles, using Lemma \ref{L2}, Theorem \ref{T1} has been extended in \cite{pal0} to obtain more circulant graphs on $2^k$ vertices exhibiting PGST. Also, there we find some circulant graphs which are almost periodic. We include the result for convenience.  

\begin{theorem}\label{T2}\cite{pal0}
Let $k\in\Nl$ and $n=2^k$. Also let $Cay\left(\Zl_n,S\right)$ be a non-integral circulant graph. Let $d$ be the least among all the divisors of $n$ so that $S\cap S_n(d)$ is a non-empty proper subset of $S_n(d)$. If $\left|S\cap S_n(d)\right|\equiv 2\pmod{4}$ then $Cay\left(\Zl_n,S\right)$ admits pretty good state transfer with respect to a sequence in $2\pi\Zl$. Moreover, if $\left|S\cap S_n(d)\right|\equiv 0\pmod{4}$ then $Cay\left(\Zl_n,S\right)$ is almost periodic with respect to a sequence in $2\pi\Zl$.
\end{theorem}
Note that Theorem \ref{T2} fails to decide whether $Cay\left(\Zl_n,S\right)$ admits PGST whenever $\left|S\cap S_n(d)\right|\equiv 0\pmod{4}$. Computational results suggest that few such graphs exhibits PGST. Consider the following example.
\begin{ex}\label{Ex1}
Let $n=16,~S=\{1,2,3,4,12,13,14,15\}$ and consider the circulant graph $G=Cay\left(\Zl_n,S\right)$. Observe that $S\cap S_n(1)=\{1,3,13,15\}\subsetneq S_n(1).$ Thus, when considering PGST, Theorem \ref{T2} is inconclusive in this case. However, calculating the $\left(0,\frac{n}{2}\right)$-th entry of the transition matrix of $G$, we obtain the following outcome presented in Figure \ref{f:p1} with the help of GNU Octave \cite{eat}. Later we shall show that $G$ exhibits PGST indeed.
\begin{figure}[h!]
\begin{subfigure}{0.4\textwidth}
\begin{tikzpicture}[thin, scale=0.7, every node/.style={scale=0.8}]
\grCirculant{16}{1,15,3,13,2,14,4,12}
\end{tikzpicture}
\end{subfigure}
\hfill
\begin{subfigure}{0.6\textwidth}
\includegraphics[width=25em]{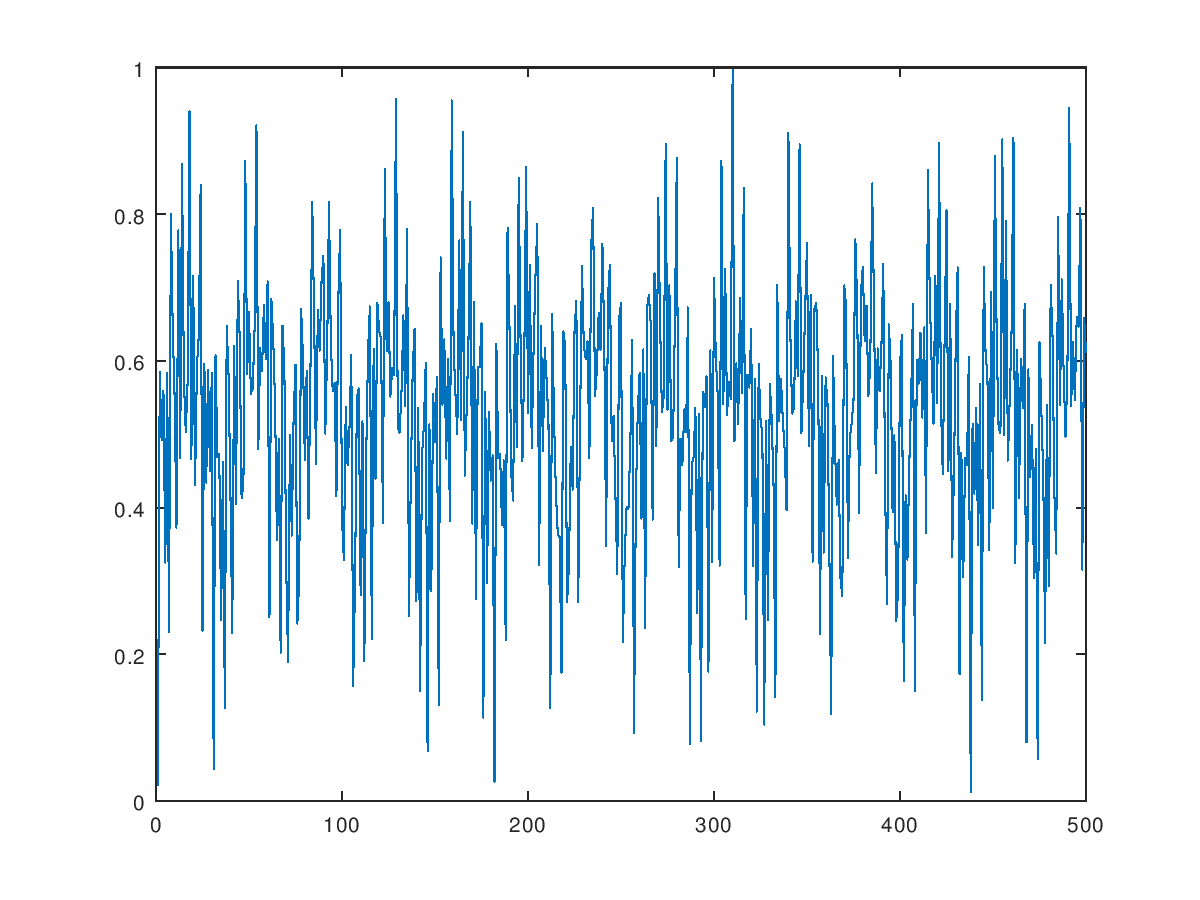}
\end{subfigure}
\caption{\small The $\left(0,\frac{n}{2}\right)$-th entry of transition matrix of $Cay\left(\Zl_n,S\right)$ at $[2*7500*\pi,~2*8000*\pi]\cap 2\pi\Zl$.}
\label{f:p1}
\end{figure}
\end{ex}

Now we shall prove the following result, revisiting the proof of Lemma[6] in \cite{pal4}, on the eigenvalues of circulant graphs. 
\begin{lem}\label{L3}
Let $k,k'\in\Nl$ be such that $n=2^k$ and $d=2^{k'}$ with $\frac{n}{d}\geq 8$. Let $\lambda_l,$ where $l=0,1,\ldots,n-1$, be the eigenvalues of $C_n$. For $\delta>0$ there exists $t\in 2\pi\Zl$ so that for each $l$ there exists an integer $l'$ such that 
\begin{eqnarray*}
\begin{cases}
\left|\left(\lambda_l t+a\pi\right)- 2l'\pi\right|<\frac{\delta}{n}, & \text{ if $l=2^{k'}\cdot a$ and $a$ is odd.}\\
\left|\lambda_l t- 2l'\pi\right|<\frac{\delta}{n}, & \text{ otherwise.}
\end{cases}
\end{eqnarray*}
\end{lem}
\begin{proof}
We first show that the distinct positive eigenvalues of $C_n$ are linearly independent over $\Ql$. The eigenvalues of $C_n$ can be realized as $\lambda_l=\omega_n^l+\omega_n^{-l}.$ Note that for $ 1\leq l\leq 2^{k-2}-1$, the relation $\lambda_l=-\lambda_{\frac{n}{2}-l}=-\lambda_{\frac{n}{2}+l}=\lambda_{n-l}$ holds, and the remaining eigenvalues $\lambda_0,\lambda_{2^{k-2}},\lambda_{2^{k-1}}$ and $\lambda_{3\cdot2^{k-2}}$ are integers. It is well known that the minimal polynomial of $\omega_n$ over $\Ql$ has degree $\phi(n)$, where $\phi$ is Euler's phi-function (see \cite{esc}). The distinct positive eigenvalues are $\lambda_l$ where $0\leq l<\frac{n}{4}.$ If the distinct positive eigenvalues are dependent over $\Ql$ then $\omega_n$ will be a root of a polynomial of degree at most $m=2^{k-1}-1$ as $\omega_n^{-1}=-\omega_n^{m}$. But since $2^{k-1}-1<2^{k-1}= \phi(2^k),$ we conclude that the distinct positive eigenvalues are linearly independent over $\Ql$. Now, for $ 1\leq l\leq 2^{k-2}-1$, assume that
\begin{eqnarray*}
\alpha_l=\begin{cases} \frac{1}{2}, & \text{ if $l=2^{k'}\cdot a$ and $a$ is odd.}\\
0, & \text{ otherwise.}
\end{cases}
\end{eqnarray*}
By Kronecker approximation theorem, for $\delta>0$ there exist $q,m_1,\ldots,m_{2^{k-2}-1}\in\Zl$ such that for $l=1,\ldots, 2^{k-2}-1$
\begin{eqnarray}\label{L3E1}
\left|q\lambda_l-m_l-\alpha_l\right|<\frac{\delta}{2n\pi},\; \emph{i.e,}\;\left|2q\pi\lambda_l-2m_l\pi-2\alpha_l\pi\right|<\frac{\delta}{n}.
\end{eqnarray} 
Since $\lambda_0,\lambda_{2^{k-2}},\lambda_{2^{k-1}}$ and $\lambda_{3\cdot2^{k-2}}$ are integers, considering $t=2q\pi$, we observe that for each $l=0,\;2^{k-2},\;2^{k-1},\;3\cdot2^{k-2}$ there exists an integer $l'$ such that
\[\left|\lambda_l t- 2l'\pi\right|<\frac{\delta}{n}.\]
Also for each $l=1,\ldots, 2^{k-2}-1$, considering $t=2q\pi$ and using Equation \ref{L3E1}, we find
\begin{eqnarray*}
l'=\begin{cases} \frac{2m_l+a+1}{2}, & \text{ if $l=2^{k'}\cdot a$ and $a$ is odd.}\\
m_l, & \text{ otherwise.}
\end{cases}
\end{eqnarray*}
such that
\begin{eqnarray*}
\begin{cases}
\left|\left(\lambda_l t+a\pi\right)- 2l'\pi\right|<\frac{\delta}{n}, & \text{ if $l=2^{k'}\cdot a$ and $a$ is odd.}\\
\left|\lambda_l t- 2l'\pi\right|<\frac{\delta}{n}, & \text{ otherwise.}
\end{cases}
\end{eqnarray*}
Since $\lambda_l=-\lambda_{\frac{n}{2}-l}=-\lambda_{\frac{n}{2}+l}=\lambda_{n-l}$ holds for $ 1\leq l\leq 2^{k-2}-1$, considering $t=2q\pi$, we conclude that for each $l=0,\ldots, n-1,$ there is an integer $l'$ such that
\begin{eqnarray*}
\begin{cases}
\left|\left(\lambda_l t+a\pi\right)- 2l'\pi\right|<\frac{\delta}{n}, & \text{ if $l=2^{k'}\cdot a$ and $a$ is odd.}\\
\left|\lambda_l t- 2l'\pi\right|<\frac{\delta}{n}, & \text{ otherwise.}
\end{cases}
\end{eqnarray*}
This completes the proof.
\end{proof}
It is well known that if a circulant graph exhibits PGST then it is almost periodic. In Theorem \ref{T2}, we find that if $\left|S\cap S_n(d)\right|\equiv 0\pmod{4}$ then the graph $Cay\left(\Zl_n,S\right)$ is almost periodic, and thus presumably many such graphs exhibits PGST. We present a subclass of those circulant graphs admitting PGST extending Theorem \ref{T2}.
\begin{theorem}\label{T3}
Let $k\in\Nl$ and $n=2^k$. Also let $Cay\left(\Zl_n,S\right)$ be a non-integral circulant graph. Let $d$ be the least among all the divisors of $n$ so that $S\cap S_n(d)\neq \emptyset$ and $S\cap S_n(d)\subsetneq S_n(d)$ with $\left|S\cap S_n(d)\right|\equiv 2\pmod{4}$. Then $Cay\left(\Zl_n,S\right)$ admits pretty good state transfer with respect to a sequence in $2\pi\Zl$.
\end{theorem}

\begin{proof}
First we shall prove that there exists a sequence in $2\pi\Zl$ with respect to which $Cay\left(\Zl_n,S\cap S_n(d)\right)$ exhibits PGST and $Cay\left(\Zl_n,S\setminus S_n(d)\right)$ is almost periodic. Let us suppose $d=2^{k'}.$ Since $S\cap S_n(d)\neq \emptyset$ and $S\cap S_n(d)\subsetneq S_n(d)$, the set $S_n(d)$ must have atleast $4$ elements. Therefore we must have $\phi\left(\frac{n}{d}\right)\geq 4,$ i.e, $\frac{n}{d}\geq 8$.\\
The exponential function $\exp{\left(-ix\right)}$ is uniform continuous. Therefore, for $\epsilon>0$, there exists $\delta>0$ such that $x_1,x_2\in\Rl$ and $|x_1-x_2|<\delta$ implies $|\exp{\left(-ix_1\right)}-\exp{\left(-ix_2\right)}|<\epsilon.$ Now we proceed using the following steps.\\
\textbf{Step 1:} If $\theta_l$ denotes the eigenvalues of $Cay\left(\Zl_n,S\cap S_n(d)\right)$ then
\[\theta_l=\frac{1}{2}\sum\limits_{s\in S\cap S_n(d)}\lambda_{ls}.\]
Since $\left|S\cap S_n(d)\right|\equiv 2\pmod{4}$, using Lemma \ref{L3} and triangle inequality, we obtain that for $\delta>0$ there exists $t\in 2\pi\Zl$ such that for each $l$ there exists an integer $l'$ such that 
\begin{eqnarray*}
\left|\left(\theta_l t+l\pi\right)- 2l'\pi\right|<\delta
\end{eqnarray*}
Hence there exists $t=2\pi\Zl$ such that $\left|\exp{\left[-i\left(\theta_l t+l\pi\right)\right]}- 1\right|<\epsilon.$ Finally, by Equation \ref{E2}, if $H_{S\cap S_n(d)}(t)$ is the transition matrix of $Cay\left(\Zl_n,S\cap S_n(d)\right)$ then
\begin{eqnarray*}
\left|\left[H_{S\cap S_n(d)}(t)\right]_{0,\frac{n}{2}}-1\right|=\frac{1}{n}\left|\sum\limits_{l=0}^{n-1}\left(\exp{\left[-i\left(\theta_l t+l\pi\right)\right]}-1\right)\right|<\epsilon.
\end{eqnarray*}
This leads to the conclusion that $Cay\left(\Zl_n,S\cap S_n(d)\right)$ admits PGST with respect to a sequence $\left\lbrace t_m\right\rbrace$ (say) in $2\pi\Zl$. Hence
\begin{eqnarray}\label{T3E1}
\lim_{m\rightarrow\infty} H_{S\cap S_n(d)}\left(t_m\right) {\bf e}_0={\bf e}_{\frac{n}{2}}
\end{eqnarray}
\textbf{Step 2:} If $d'$ be a divisor of $n$ such that $d'<d$ and $S\cap S_n(d')\neq \emptyset$ then we have $\left|S\cap S_n(d')\right|\equiv 0\pmod{4}.$ If $\eta_l$ denotes the eigenvalues of $Cay\left(\Zl_n,S\cap S_n(d')\right)$ then
\[\eta_l=\frac{1}{2}\sum\limits_{s\in S\cap S_n(d')}\lambda_{ls}.\]
Since $\left|S\cap S_n(d)\right|\equiv 0\pmod{4}$, using Lemma \ref{L3} and triangle inequality, we can carefully choose the same $t\in 2\pi\Zl$ as in Step 1 so that for each $l$ there exists an integer $l'$ such that 
\begin{eqnarray*}
\left|\eta_l t- 2l'\pi\right|<\delta
\end{eqnarray*}
Hence there exists $t=2\pi\Zl$ such that $\left|\exp{\left[-i\eta_l t\right]}- 1\right|<\epsilon.$ Finally, if $H_{S\cap S_n(d')}(t)$ is the transition matrix of $Cay\left(\Zl_n,S\cap S_n(d')\right)$ then Equation \ref{E2a} implies
\[\left|\left[H_{S\cap S_n(d')}(t)\right]_{0,0}-1\right|=\frac{1}{n}\left|\sum\limits_{l=0}^{n-1}\left(\exp{\left[-i\eta_l t\right]}-1\right)\right|<\epsilon.\]
This leads to the conclusion that $Cay\left(\Zl_n,S\cap S_n(d')\right),$ where $d'<d$, is almost periodic with respect to the same sequence $\left\lbrace t_m\right\rbrace$ as obtained in Step 1.  Therefore we have
\begin{eqnarray}\label{T3E2}
\lim\limits_{m\rightarrow\infty} H_{S\cap S_n(d')}\left(t_m\right)=I.
\end{eqnarray}
\textbf{Step 3:} If $d''$ be a divisor of $n$ such that $d''>d$ and $S\cap S_n(d'')\neq \emptyset$ then for $s\in S\cap S_n(d'')$ the size of each disjoint cycles appearing in $Cay\left(\Zl_n,\left\lbrace s,n-s\right\rbrace\right)$ is strictly less than the size of the cycle $C_{n/d}$. In Step 1, in particular, if we consider $S\cap S_n(d)=\{d,-d\}$ then $Cay\left(\Zl_n,S\cap S_n(d)\right)$, which is disjoint union of $d$ copies of the cycle $C_{n/d}$, admits PGST with respect to $\left\lbrace t_m\right\rbrace$. Hence the cycle $C_{n/d}$, in particular, admits PGST with respect to $\left\lbrace t_m\right\rbrace$. Therefore, by Lemma \ref{L2}, we have a subsequence of $\left\lbrace t_m\right\rbrace$ with respect to which each component of $Cay\left(\Zl_n,\left\lbrace s,n-s\right\rbrace\right),~s\in S\cap S_n(d''),$ is almost periodic. Hence $Cay\left(\Zl_n,\left\lbrace s,n-s\right\rbrace\right)$ is also almost periodic. Since $s\in S\cap S_n(d'')$ is finite, we can appropriately choose a subsequence $\left\lbrace t_{m_r}\right\rbrace$ of $\left\lbrace t_m\right\rbrace$ such that for all $s\in S\cap S_n(d'')$, the graph $Cay\left(\Zl_n,\left\lbrace s,n-s\right\rbrace\right)$ is almost periodic with respect to $\left\lbrace t_{m_r}\right\rbrace.$ Hence, by Proposition \ref{IP1}, the graph $Cay\left(\Zl_n,S\cap S_n(d'')\right)$ is almost periodic with respect to $\left\lbrace t_{m_r}\right\rbrace.$ Therefore, if $H_{S\cap S_n(d'')}(t)$ is the transition matrix of $Cay\left(\Zl_n,S\cap S_n(d'')\right)$ then
\begin{eqnarray}\label{T3E3}
\lim\limits_{m\rightarrow\infty} H_{S\cap S_n(d'')}\left(t_{m_r}\right)=I.
\end{eqnarray}
Notice that $S_n(d_1)\cap S_n(d_2)=\emptyset$ whenever $d_1\neq d_2$. Therefore, by Proposition \ref{IP1} and using Equation \ref{T3E2} and Equation \ref{T3E3}, we obtain
\begin{eqnarray}\label{T3E4}
H_{S}\left(t_{m_r}\right)
&=& \left[\prod\limits_{d'|n,~d'<d} H_{S\cap S_n(d')}\left(t_{m_r}\right)\right]\cdot\left[\prod\limits_{d''|n,~d<d''} H_{S\cap S_n(d'')}\left(t_{m_r}\right)\right]\cdot H_{S\cap S_n(d)}\left(t_{m_r}\right) \nonumber \\
&=& H_{S\cap S_n(d')}\left(t_{m_r}\right).
\end{eqnarray}
Finally, by Equation \ref{T3E1} and Equation \ref{T3E4}, we conclude that $Cay\left(\Zl_n,S\right)$ admits PGST with respect to a sequence in $2\pi\Zl$.
\end{proof}
Recall that, in Example \ref{Ex1}, the graph $Cay\left(\Zl_n,S\right)$ satisfying all the conditions of Theorem \ref{T3} as $S\cap S_n(2)=\{2,4\}$. Hence $Cay\left(\Zl_n,S\right)$ admits PGST with respect to a sequence in $2\pi\Zl$. Using Theorem \ref{T3}, we can find many such graphs exhibiting PGST. However notice that if, for all $d$ with $S\cap S_n(d)\subsetneq S_n(d)$, we have $\left|S\cap S_n(d)\right|\equiv 0\pmod{4}$ then Theorem \ref{T3} does not apply. Consider the following example.
\begin{ex}\label{Ex2}
Let $n=16,~S=\{1,3,4,12,13,15\}$ and consider the circulant graph $G=Cay\left(\Zl_n,S\right)$. Observe that $S\cap S_n(1)=\{1,3,13,15\}\subsetneq S_n(1).$ Obtaining the $\left(0,\frac{n}{2}\right)$-th entry of the transition matrix of $G$, we observe the following Figure \ref{f:p2} in GNU Octave \cite{eat}. The next result suggests that $G$ exhibits PGST.
\begin{figure}[h!]
\begin{subfigure}{0.4\textwidth}
\begin{tikzpicture}[thin, scale=0.7, every node/.style={scale=0.8}]
\grCirculant{16}{1,3,4,12,13,15}
\end{tikzpicture}
\end{subfigure}
\hfill
\begin{subfigure}{0.6\textwidth}
\includegraphics[width=25em]{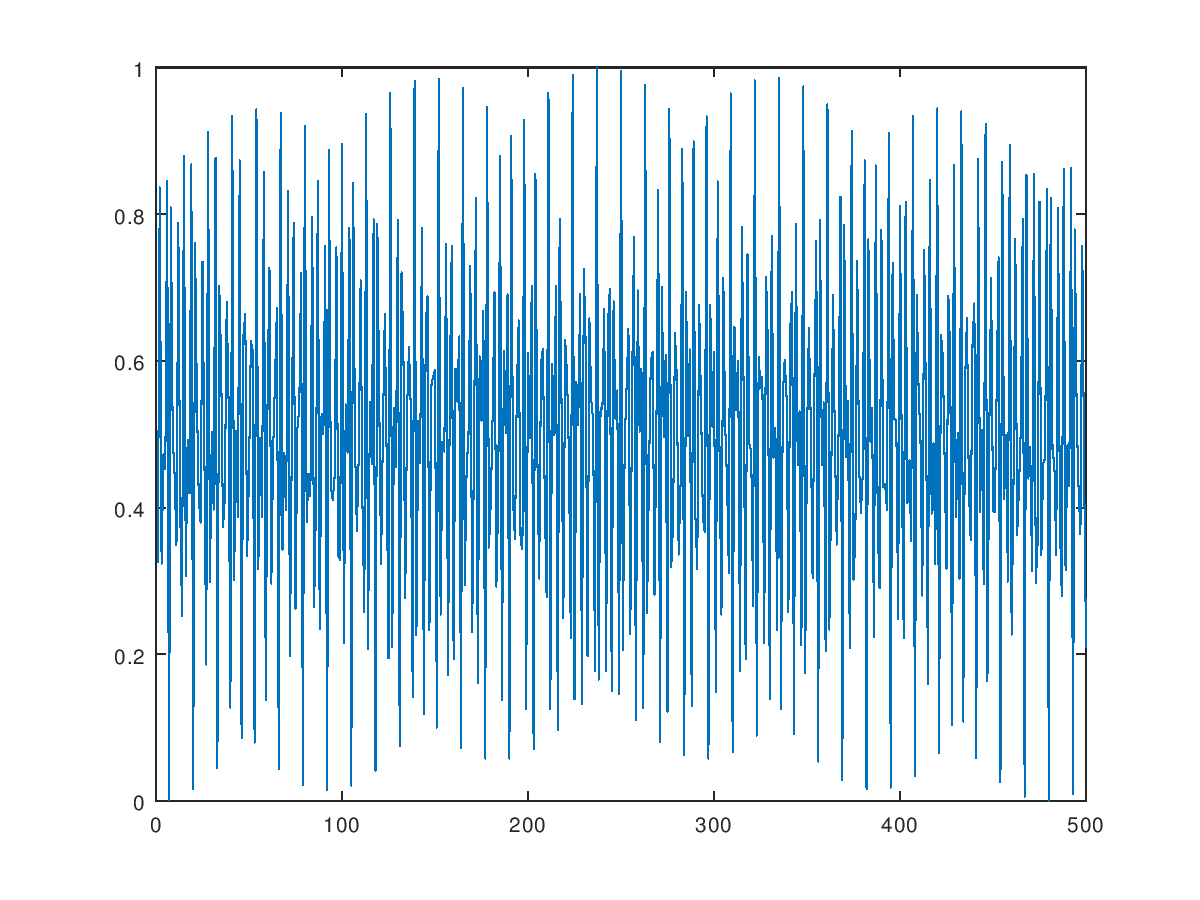}
\end{subfigure}
\caption{\small The $\left(0,\frac{n}{2}\right)$-th entry of transition matrix of $Cay\left(\Zl_n,S\right)$ at $[0,500\pi]\cap \left(2\Zl+1\right)\frac{\pi}{2}$.}
\label{f:p2}
\end{figure}
\end{ex}
In such cases mentioned above, we show that if either $\frac{n}{2}$ or $\frac{n}{4}$ but not both is in $S$ then $Cay\left(\Zl_n,S\right)$ admits PGST. Consider the following theorem.
\begin{theorem}\label{T4}
Let $k\in\Nl$, $n=2^k$ and $Cay\left(\Zl_n,S\right)$ be a circulant graph with $\left|\left\lbrace\frac{n}{2},~\frac{n}{4}\right\rbrace \cap S\right|\equiv 1\pmod{2}$. If each divisor $d~\left(\neq \frac{n}{2},\frac{n}{4}\right)$ of $n$ satisfy $\left|S\cap S_n(d)\right|\equiv 0\pmod{4}$ then $Cay\left(\Zl_n,S\right)$ admits pretty good state transfer with respect to a sequence in $\left(2\Zl+1\right)\frac{\pi}{2}$.
\end{theorem}

\begin{proof}
We have established in Lemma \ref{L3} that the distinct positive eigenvalues of $C_n$ are linearly independent over $\Ql$. For $ 1\leq l\leq 2^{k-2}-1$, let us choose $\alpha_l=-\frac{\lambda_l}{2}$. By Kronecker approximation theorem, for $\delta>0$ there exist $q,m_1,\ldots,m_{2^{k-2}-1}\in\Zl$ such that for $l=1,\ldots, 2^{k-2}-1$
\begin{eqnarray}\label{T4E1}
\left|q\lambda_l-m_l-\alpha_l\right|<\frac{\delta}{n\pi},\; \emph{i.e,}\;\left|\left((2q+1)\frac{\pi}{2}\right)\lambda_l-m_l\pi\right|<\frac{\delta}{n}.
\end{eqnarray} 
Since $\lambda_0,\lambda_{2^{k-2}},\lambda_{2^{k-1}}$ and $\lambda_{3\cdot2^{k-2}}$ are even integers, considering $t=\left(2q+1\right)\frac{\pi}{2}$, we observe that for each $l=0,\;2^{k-2},\;2^{k-1},\;3\cdot2^{k-2}$ there exists an integer $l'$ such that
\[\left|\lambda_l t- l'\pi\right|<\frac{\delta}{n}.\]
Since $\lambda_l=-\lambda_{\frac{n}{2}-l}=-\lambda_{\frac{n}{2}+l}=\lambda_{n-l}$ holds for $ 1\leq l\leq 2^{k-2}-1$, considering $t=\left(2q+1\right)\frac{\pi}{2}$ and using Equation \ref{T4E1}, we conclude that for each $l=0,\ldots, n-1,$ there is an integer $l'$ such that
\begin{eqnarray}\label{T4E2}
\left|\lambda_l t- l'\pi\right|<\frac{\delta}{n}.
\end{eqnarray}
Consider $S'=S\setminus\left\lbrace\frac{n}{4},\frac{3n}{4},\frac{n}{2}\right\rbrace$. If $\theta_l$ denotes the eigenvalues of $Cay\left(\Zl_n,S'\right)$ then
\[\theta_l=\frac{1}{2}\sum\limits_{s\in S'}\lambda_{ls}.\]
Since each divisor $d~\left(\neq \frac{n}{2},\frac{n}{4}\right)$ of $n$ satisfy $\left|S\cap S_n(d)\right|\equiv 0\pmod{4}$, using Equation \ref{T4E2} and triangle inequality, we obtain that for $\delta>0$ there exists $t\in\left(2\Zl+1\right)\frac{\pi}{2}$ so that for each $l$ there exists an integer $l'$ such that 
\begin{eqnarray*}
\left|\theta_l t- 2l'\pi\right|<\delta
\end{eqnarray*}
Since the exponential function is uniform continuous, for $\epsilon>0$ there exists $t\in\left(2\Zl+1\right)\frac{\pi}{2}$ such that $\left|\exp{\left[-i\theta_l t\right]}- 1\right|<\epsilon.$ Finally, if $H_{S'}(t)$ is the transition matrix of $Cay\left(\Zl_n,S'\right)$ then by Equation \ref{E2a}, we obtain
\begin{eqnarray*}
\left|\left[H_{S'
}(t)\right]_{0,0}-1\right|=\frac{1}{n}\left|\sum\limits_{l=0}^{n-1}\left(\exp{\left[-i\theta_l t\right]}-1\right)\right|<\epsilon.
\end{eqnarray*}
This leads to the conclusion that $Cay\left(\Zl_n,S'\right)$ almost periodic with respect to a sequence $\left\lbrace t_m\right\rbrace$ (say) in $\left(2\Zl+1\right)\frac{\pi}{2}$. Hence
\begin{eqnarray}\label{T4E3}
\lim_{m\rightarrow\infty} H_{S'}\left(t_m\right)=I.
\end{eqnarray}
By Proposition \ref{IP1} and using Equation \ref{T4E3}, we find that if $H_{S}(t)$ and $H_{S\setminus S'}(t)$ are the transition matrices of $Cay\left(\Zl_n,S\right)$ and $Cay\left(\Zl_n,S\setminus S'\right)$, respectively, then
\begin{eqnarray}\label{T4E4}
\lim_{m\rightarrow\infty} H_{S}\left(t_m\right)=\lim_{m\rightarrow\infty} H_{S\setminus S'}\left(t_m\right).
\end{eqnarray}
Notice that $Cay\left(\Zl_n,\left\lbrace\frac{n}{2}\right\rbrace\right)$ is the disjoint union of path on two vertices (denoted by $P_2$) and $Cay\left(\Zl_n,\left\lbrace\frac{n}{4},\frac{3n}{4}\right\rbrace\right)$ is the disjoint union of $C_4$. It is well known that both $P_2$ and $C_4$ admits perfect state transfer at $\frac{\pi}{2}$ and are periodic at $\pi$. Hence, by Equation \ref{T4E4}, we conclude that $Cay\left(\Zl_n,S\right)$ exhibits PGST with respect to a sequence in $t\in\left(2\Zl+1\right)\frac{\pi}{2}$.
\end{proof}
In Theorem \ref{T4}, if $Cay\left(\Zl_n,S\right)$ is integral then it is certainly periodic at $2\pi.$ Therefore, by Proposition \ref{IP2}, if $Cay\left(\Zl_n,S\right)$ exhibits PGST then $Cay\left(\Zl_n,S\right)$ also exhibits PST. A complete characterization of circulant graphs exhibiting perfect state transfer is appearing in \cite{mil4}. Notice that Theorem \ref{T4} in fact reveals all integral circulant graphs on $2^k$ vertices exhibiting PST. We present the scenario with an example. 
\begin{ex}\label{Ex3}
Consider $S=\{1,2,3,5,6,7\}$ and $G=Cay\left(\Zl_8,S\right).$ By Theorem \ref{so}, the graph $G$ is integral and hence it is periodic at $2\pi$ (see the spectral decomposition of the transition matrix). Also applying Theorem \ref{T4}, we find that $G$ exhibits PGST. Finally, by Proposition \ref{IP2}, we conclude that $G$ exhibits perfect state transfer.
\end{ex}

Note that if a circulant graph exhibits PGST then it is almost periodic. In Theorem \ref{T2}, we find that if $\left|S\cap S_n(d)\right|\equiv 0\pmod{4}$ then $Cay\left(\Zl_n,S\right)$ is almost periodic. It is therefore tempting to presume that all such graph exhibits PGST, however, that is not the case. We demonstrate this with the following example. 
\begin{ex}\label{Ex3}
Consider $S=\{1,7,9,15\}$ and $G=Cay\left(\Zl_{n},S\right),$ with $n=16$. Let the eigenvalues of $G$ be $\theta_l,$ where $l=0,1,\ldots,n-1$. Note here that $\theta_1=\theta_4=0$. If $G$ admits PGST then we have a sequence of real numbers $\left\lbrace t_k\right\rbrace$ and a complex number $\gamma$ with $|\gamma|=1$ such that (using Equation \ref{E2})
\[\lim_{k\rightarrow\infty}\sum\limits_{l=0}^{n-1}\exp{\left[-i\theta_l t_k+l\pi\right]}=n\gamma.\]
Since the unit circle is compact, we have a subsequence $\left\lbrace t'_k\right\rbrace$ of $\left\lbrace t_k\right\rbrace$ such that
\[\lim_{k\rightarrow\infty}\exp{[-i\theta_l t'_k+l\pi]}=\gamma,\text{ for }0\leq l\leq n-1.\]
\end{ex}
Now $\theta_1=0$ gives $\gamma=-1$ and $\theta_4=0$ gives $\gamma=1$, which is absurd. Hence the graph $G$ does not exhibit PGST.
\section*{Acknowledgements}
We are deeply grateful to Dr. Himadri Nayak, IIIT Bhagalpur, India, for the help with GNU Octave programming. Also we are indebted to the anonymous reviewer(s) for carefully assessing the manuscript.

\end{document}